\documentclass[11pt]{article}
\usepackage[a4paper]{anysize}\marginsize{3.5cm}{3.5cm}{1cm}{2cm}
\pdfpagewidth=\paperwidth \pdfpageheight=\paperheight
\usepackage{amsfonts,amssymb,amsthm,amsmath,eucal,tabu}
\usepackage{comment}
\usepackage{pgf}
\usepackage{bbm}
\usepackage{hyperref}
\theoremstyle{plain}
\newtheorem{thm}{Theorem}[section]
\newtheorem{theorem}[thm]{Theorem}

\newtheorem{lemma}[thm]{Lemma}
\newtheorem{proposition}[thm]{Proposition}
\newtheorem{corollary}[thm]{Corollary}
\newtheorem{conjecture}[thm]{Conjecture}

\theoremstyle{definition}
\newtheorem{definition}[thm]{Definition}
\newtheorem{remark}[thm]{Remark}
\newtheorem{example}[thm]{Example}

\newtheorem{problem}[thm]{Problem}

\newtheorem{thevarthm}[thm]{\varthmname}

\newenvironment{varthm*}[1]{\trivlist\item[]{\bf #1.}\it}{\endtrivlist}

\newcommand\C{\mathbb C}

\newcommand\R{\mathbb R}
\newcommand\K{\mathbb K}
\renewcommand\P{\mathbb P}
\newcommand\call{{\mathcal L}}
\newcommand\calp{{\mathcal P}}
\newcommand\newop[2]{\def#1{\mathop{\rm #2}\nolimits}}
\newop\mult{m}
\newop\mmult{m}
\newcommand\eqnref[1]{(\ref{#1})}
\newcommand\wtilde[1]{\widetilde{#1}}

\newcommand\beginproof[1]{\trivlist\item[\hskip\labelsep{\em #1.}]}

\newcommand\proofof[1]{\beginproof{Proof of #1}}
\def\endproof{\hspace*{\fill}\endproofsymbol\endtrivlist}

\def\endproofsymbol{\frame{\rule[0pt]{0pt}{6pt}\rule[0pt]{6pt}{0pt}}}

\def\keywordname{{\bfseries Keywords}}%
\def\keywords#1{\par\addvspace\medskipamount{\rightskip=0pt plus1cm
\def\and{\ifhmode\unskip\nobreak\fi\ $\cdot$
}\noindent\keywordname\enspace\ignorespaces#1\par}}
\def\subclassname{{\bfseries Mathematics Subject Classification
(2000)}\enspace}
\def\subclass#1{\par\addvspace\medskipamount{\rightskip=0pt plus1cm
\def\and{\ifhmode\unskip\nobreak\fi\ $\cdot$
}\noindent\subclassname\ignorespaces#1\par}}

\begin{document}

\author{Th.~Bauer, S.~Di Rocco, B.~Harbourne, J.~Huizenga, \\ A.~Lundman, P.~Pokora, T.~Szemberg}
\title{Bounded Negativity and Arrangements of Lines}

\date{\today}

\maketitle

\thispagestyle{empty}

\begin{abstract}
   The Bounded Negativity Conjecture predicts that for any smooth complex
   surface $X$ there exists a lower bound for the selfintersection of reduced
   divisors on $X$. This conjecture is open. It is also not known if the
   existence of such a lower bound is invariant in the birational equivalence
   class of $X$. In the present note we introduce certain constants $H(X)$ which
   measure in effect the variance of the lower bounds in the birational equivalence class
   of $X$. We focus on rational surfaces and relate the value of $H(\P^2)$ to
   certain line arrangements. Our main result is Theorem \ref{thm:main}
   and the main open challenge is Problem \ref{pro:always -4}.
\end{abstract}

\maketitle


\section{Introduction}\label{sect:intro}
   In recent years there has been growing interest in constraints on negative
   curves on algebraic surfaces \cite{RDLS, NCAS, CS14, CilRou14, Har10, KM14, MT14}.
   The Bounded Negativity Conjecture (BNC for short)
   is probably the most intriguing open question in this area, see for example
   \cite[Conjecture 1.2.1]{Har10}, \cite[Conjecture 1.1]{NCAS}.
\begin{conjecture}[Bounded Negativity Conjecture]\label{conj:bnc}
   For every smooth projective surface $X$, there exists an integer $b(X)$
   such that $C^2\geq -b(X)$ for every reduced curve $C\subset X$.
\end{conjecture}
   This Conjecture is well known to be false in finite characteristic. In the present
   paper we work therefore in the setting of complex algebraic varieties,
   where it remains open. Subsection \ref{subsec:real projective plane} is
   devoted to the special case of the real projective plane.

  Conjecture \ref{conj:bnc} is related to a number of interesting questions.
   The present note
   is motivated by the following problem.
\begin{problem}[Birational invariance of the BNC]\label{pro:bnc birat inv}
   Let $X$ and $Y$ be birationally equivalent projective surfaces. Does
   BNC hold for $X$ if and only if it holds for $Y$?
   In other words: is the bounded negativity property a birational invariant?
\end{problem}
\begin{remark}
   Note, that a solution to the above problem is not known even
   if $Y$ is the blow-up of $X$ in a single point.
\end{remark}
   Of course, if BNC is true in general, then the above problem
   has an affirmative solution. However, even in that situation,
   it is still of interest to know how the bounds $b(X)$ and $b(Y)$
   are related in terms of the complexity of a birational map between
   $X$ and $Y$.

   In the present note we study Problem \ref{pro:bnc birat inv}
   for blow-ups $Y_s$ of $\P^2$ in \emph{arbitrary} sets of $s$ points.
   A recent preprint by Ciliberto and Roulleau \cite{CilRou14}
   addresses BNC on blow-ups of $\P^2$ in \emph{general} points.
   These two set-ups are quite different. It is predicted by the
   Segre-Harbourne-Gimigliano-Hirschowitz (SHGH) Conjecture that $(-1)$-curves are the
   only negative curves on blow-ups of $\P^2$ in general
   points, independent of their number $s$. On the other hand, if one is allowed to pick any $s$ points,
   it is elementary to see that one can obtain reduced curves $C$ for which
   $C^2$ is arbitrarily negative by allowing $s$ to grow.
   This raises the question of the boundedness of $C^2/s$.
   Our main result, Theorem \ref{thm:main}, shows that $C^2/s>-4$
   for reduced curves $C$ on $Y_s$ which are strict transforms of
   configurations of lines in $\P^2$. In fact, we know of no examples of
   irreducible curves $C\subset Y_s$ for which $C^2/s\leq -2$ and no examples of any
   kind for which $C^2\leq -4$. This has led to our expectation
   that the bound $C^2/s>-4$ which applies for curves coming
   from configurations of lines may apply in all cases,
   but we have not been able to prove this so far; see Problem \ref{pro:always -4}.
   We also include a short discussion of the only two
   line arrangements we know of for which $C^2/s\leq -3$.

\section{Local negativity}\label{sect:local negativity}
   Let $X$ be a smooth projective surface and let $C$ be
   a reduced curve on $X$.
   For any point $P\in X$, let $\mult_{P_i}(C)$ denote the multiplicity of $C$ at $P$; i.e.,
   if $h\in {\mathcal O}_{X,P}$ is a local equation for $C$ at $P$,
   then $\mult_{P_i}(C)$ is the largest $t$ such that $h\in {\frak m}^t$, where $\frak m$
   is the maximal ideal of ${\mathcal O}_{X,P}$. Alternatively, if $f:Y\to X$ is the blow up of $X$ at $P$
   and $E_P$ its exceptional divisor, then $t$ is the multiplicity of the component of $E_P$
   occurring in the total transform of $C$ on $Y$ and hence the proper transform $\wtilde{C}$ of $C$ on $Y$
   (up to linear equivalence) is $f^*C-tE_P$.

   We note that it is easy to find surfaces birational to a given surface and
   carrying arbitrarily negative curves.
   For example, let $P_1,\ldots,P_s$ be mutually distinct smooth
   points on $C$ and take $f:Y\to X$ to be the blow-up of $X$ at the
   points $P_1,\ldots,P_s$ with exceptional divisors $E_1,\ldots,E_s$.
   Then the proper transform of $C$ is
\begin{equation}\label{smoothpoints}
   \wtilde{C}=f^*C-\sum_{i=1}^s E_i
\end{equation}
   and has $\wtilde{C}^2=C^2-s$. As Remark \ref{singptssuffice} makes clear, we can ignore trivial
   situations of this kind by introducing the following $H$--constants,
   which measure the \emph{local negativity} of curves on surfaces (in analogy to
   the local positivity measured by Seshadri constants).
\begin{definition}[Hadean constants, or $H$--constants]\label{def:H-constants}
   Let $X$ be a smooth projective surface and let
   $\calp=\left\{P_1,\ldots,P_s\right\}$ be a set of $s\geq 1$ mutually
   distinct points of $X$. Then the \emph{$H$--constant of $X$ at $\calp$}
   is defined as
\begin{equation}\label{eq:H-const for calp}
      H(X;\calp):=\inf_C\frac{\wtilde{C}^2}{s},
\end{equation}
   where $\wtilde{C}$ is the proper transform of $C$ with respect to
   the blow-up $f:Y\to X$ of $X$ at the set $\calp$ and
   the infimum is taken over all \emph{reduced} curves $C\subset X$.
   Note that $\wtilde{C}^2=(f^*C-\sum_{i=1}^s (\mult_{P_i}(C))E_i)^2=C^2-\sum_{i=1}^s (\mult_{P_i}(C))^2$,
   where $E_1,\ldots,E_s$ are the exceptional divisors of the blown up points.
   Thus $\frac{\wtilde{C}^2}{s}$ can be thought of as the average of the numbers $-((\mult_{P_i}(C))^2-\frac{C^2}{s})$.\\
   Similarly, we define the \emph{$s$--tuple $H$--constant of $X$}
   as the infimum
   $$H(X;s):=\inf\limits_{\calp}H(X;\calp),$$
   where the infimum now is taken over all $s$--tuples of mutually
   distinct points in $X$.\\
   Finally, we define the \emph{global $H$--constant of $X$} as
   $$H(X):=\inf\limits_{s\geq 1}H(X;s).$$
\end{definition}

\begin{example}\label{-2example}
  We do not know if it ever happens that $H(X)=-\infty$, but we always have $H(X)\leq -2$.
  For example, embed $X$ in projective space as a surface of some degree $d$.
  Let $C$ be the intersection of $X$ with a general union of $r$ hyperplanes.
  Then $C$ has $r$ smooth irreducible components $C_i$, each pair of which meet in $d^2$
  mutually distinct points.
  Thus $C$ has $s=d^2\binom{r}{2}$ nodes, with each component $C_i$ passing
  through $d^2(r-1)$ of these $d^2\binom{r}{2}$ points.
  Taking $\calp$ to be these nodes, note that
   $\wtilde{C}$ has $r$ disjoint components $\wtilde{C}_i$, each of which has
   $\wtilde{C}_i^2=d^2-(r-1)d^2=(2-r)d^2$.
   This gives
   $$H(X;\calp)\leq \frac{\wtilde{C}^2}{s} = \frac{r(2-r)d^2}{d^2\binom{r}{2}}=-2+\frac{2}{r-1},$$
   and hence $H(X)\leq -2$.
\end{example}

\begin{remark}
   The relation between $H$--constants and the BNC can be explained in the following way.
   It is known that BNC holds if and only if
   it holds for irreducible curves \cite{NCAS}. Suppose that $H(X)$ is not $-\infty$.
   Then for any $s\geq 1$ and any irreducible curve $D$ on the blow-up of $X$
   at $s$ points, we have
   $$D^2\geq sH(X).$$
   Hence BNC would hold on all blow-ups of $X$ at $s$ mutually distinct points.
   On the other hand, BNC might still be true, even if $H(X)=-\infty$.
   Nonetheless, this suggests that computing $H(X)$ can be expected to be challenging.
\end{remark}

\begin{remark}\label{singptssuffice}
   Example \ref{-2example} also suggests a point of view
   focusing on the curves rather than on the points:
   given any reduced curve $C\subset X$ and a set $\calp\subset X$ of $s$ points, define
   $$H(C;\calp)=\frac{\wtilde{C}^2}{s}$$
   where $\wtilde{C}$ is as in Example \ref{-2example}, and, when $C$ is singular, define the Hadean constant of $C$ to be
   $$H(C)=\inf H(C;\calp')$$
   where the infimum is taken over all subsets $\calp'$ of the set of singular points of $C$.
   As we saw above, there are curves $C$ and point sets $\calp'$
   with $H(C;\calp')\leq-1$. For any $C$ and $\calp'$ with $H(C;\calp')\leq-1$, regarding $H(C;\calp')$ as an average,
   it is clear that $\calp'$ must contain some singular points of $C$. Let $\calp$ be the subset of $\calp'$ of
   points singular on $C$.
   Then taking $a$ to be the average of the numbers $(\mult_P(C))^2-\frac{C^2}{s}$ for $P\in \calp$ where $s=|\calp|$
   is the number of points in $\calp$, we have
   $$H(C)\leq H(C;\calp)=\frac{C^2-\sum_{P\in\calp}(\mult_P(C))^2}{s}=-a\leq-\frac{sa+t}{s+t}=H(C;\calp'),$$
   where $|\calp'|=s+t$; i.e., if $H(C;\calp')\leq-1$, one might as well use only
   the points of $\calp'$ singular on $C$. It follows that
   $$H(X)=\inf_{\text{all reduced singular $C$}}H(C).$$
\end{remark}

   By the preceding remark, $H(X)$ is the infimum of the
   Hadean constants of reduced singular curves on $X$. Given that computing $H(X)$
   is expected to be difficult, it becomes of interest to study bounds $H(X)\leq H(C)$
   given by specific reduced singular curves $C\subset X$.
   An important special case is that of $X=\P^2$;
   as Example \ref{-2example} shows it is not hard to find plane curves $C$ with $H(C)<-1$,
   and it is easy to check that the image $C$ of a general map of degree $d$ of $\P^1$ into $\P^2$
   gives an irreducible curve such that $H(C)$ approaches $-2$ from above as $d\to\infty$.
   It is much harder to find examples with $H(C)\leq -2$. In fact, as mentioned above, we do not know
   any plane curves $C$ with $H(C)\leq-4$ nor do we know
   any examples of irreducible plane curves with $H(C)\leq-2$. (Examining
   all possibilities allowed by the genus formula shows that any such
   irreducible singular plane curve would have to have degree at least 21.) Thus so far plane curves $C$ with Hadean constants
   $H(C)\leq -2$ appear to be a largely invisible underworld.
   In the next section we take the first steps toward getting to the bottom of it (if indeed it is not bottomless!) by looking at
   plane curves which are unions of lines. Restricting to configurations of lines
   may seem to be a strong restriction, but perhaps it is not; see Problem \ref{pro:always -4}.

\section{Linear local negativity}
   In this section we are interested in configurations of lines
   in the projective plane. By such a configuration we understand
   a family $\call=\left\{L_1,\ldots,L_d\right\}$ of mutually distinct
   lines $L_i$. Let $\calp(\call)=\left\{Q_1,\ldots,Q_s\right\}$ be the set
   of points in $\P^2$, where at least two of the lines in $\call$ meet;
   we call these points \emph{singular} points of the configuration
   (they are precisely the singular points of the union of the $s$ lines).
   For a point $Q\in\P^2$, denote by $\mmult_{Q}(\call)$ the number of lines in $\call$
   passing through that point (of course $\mmult_Q(\call)=\mult_Q(C)$, where $C$ is the
   union of lines in the configuration $\call$).

   Now we are in a position to introduce the following variant of Definition \ref{def:H-constants}.
\begin{definition}\label{def:H-const linear}
   Let $\calp=\left\{P_1,\ldots,P_s\right\}$ be a set of mutually
   distinct $s\geq 1$ points in the projective $\P^2$.
   Then the \emph{linear $H$--constant at $\calp$}
   is defined as
   $$H_L(\calp):=\inf\limits_{\call} H_L(\calp,\call),$$
   where for a configuration $\call$ we have
   \begin{equation}\label{eq:linear H-const for calp}
      H_L(\calp,\call)=\frac{d^2-\sum_{i=1}^{s}\mmult_{P_i}(\call)^2}{s}.
   \end{equation}
   Similarly as before, we define the \emph{$s$--tuple linear $H$--constant}
   as the infimum
   $$H_L(s):=\inf\limits_{\calp}H_L(\calp),$$
   where now the infimum is taken over all $s$--tuples of mutually
   distinct points in $\P^2$.\\
   Finally, we define the \emph{global linear $H$--constant of $\P^2$} as
   $$H_L:=\inf\limits_{s\geq 1}H_L(s).$$
\end{definition}
\begin{remark}\label{LinearCaseFullSingSet}
   Since there are arrangements $\call$ and points $\calp'$ with
   $H_L(\calp',\call)\leq-1$ (take any $r\geq3$ general lines for $\call$ and
   their singular points for $\calp'$), it follows in any such case
   as in Remark \ref{singptssuffice} that $H(C)\leq H_L(\calp,\call)\leq H_L(\calp',\call)$, where $C$ is the union of the lines in $\call$
   and $\calp$ is the subset of points in $\calp'$ singular on $C$. Thus
   $H_L=\inf H(C)$, where the infimum
   is taken over all curves $C$ which are unions of lines.
   (In fact, if $C$ is the union of lines in an arrangement $\call$ with $H(\calp(\call),\call)\leq -1$,
   it follows from Theorem \ref{thm:main} that $-4\leq\inf H(C)$ and so $H(C)=H(C;\calp(\call))=H_L(\calp(\call),\call)$,
   since a weighted average of a value $H(C)$ greater than or equal to $-4$ with
   numbers less than or equal to $-4$ (i.e., minus the squares of
   multiplicities at singular points) never is more than $H(C)$.)
\end{remark}

   Our main result is the following bound on the constant $H_L$.
   The main point is that the bound holds for an arbitrary number of points $s$.
   Moreover the bound is very explicit and close to optimal (see Remark \ref{rmk:opti}).
   This is considerable progress when compared to \cite[Section 3.8]{RDLS},
   where this subject was first taken on.
\begin{theorem}[Bounded linear negativity on $\P^2$]\label{thm:main}
   With the above notation, we have
   $$H_L \geq -4.$$
\end{theorem}
   The main ingredient in the proof of this Theorem is the following
   inequality due to Hirzebruch combined with some ad hoc arguments. For $k\geq 2$, let $t_k(\call)$
   denote the number of points where exactly $k$ lines from
   $\call$ meet.
\begin{theorem}[Hirzebruch inequality]
   Let $\call$ be an arrangement of $d$ lines in the complex projective plane $\P^2$.
   Then
   \begin{equation}\label{eq:Hirzebruch}
      t_2+\frac34t_3\geq d+\sum\limits_{k\geq 5}(k-4)t_k,
   \end{equation}
   provided $t_d=t_{d-1}=0$.
\end{theorem}
\begin{proof}
   See \cite[Section 3 and page 140]{Hir83}.
\end{proof}
\begin{remark}
   Various refinements of the above inequality are known;
   see for example formula $*(10)$ on page 141 in \cite{BHH87}.
   However they don't contribute towards improving the lower
   bound in Theorem \ref{thm:main}.
\end{remark}
\begin{remark}
   The proof of the Hirzebruch inequality is based on the logarithmic
   Miyaoka--Yau--Sakai inequality, which assumes
   the complex numbers. See \cite[Appendix]{RDLS} for a detailed proof
   and some relevant comments.
\end{remark}

\proofof{Theorem \ref{thm:main}}
To prove the theorem, we will show that for any configuration of lines $\call$ and points $\calp$ we have $H_L(\calp,\call)>-4.$
By Remark \ref{LinearCaseFullSingSet}, if for some $\call$ and $\calp'$ we have $H_L(\calp',\call)\leq -4$,
then $H_L(\calp,\call)\leq H_L(\calp',\call)$, where $\calp$ is the subset of $\calp'$ of points singular for $\call$.
This reduces us to considering the case that
$\calp\subseteq\calp(\call)$. But in that case we would also have $H_L(\calp(\call),\call)\leq -4$, since $H_L(\calp(\call),\call)$ is a weighted average
of $H_L(\calp,\call)$ with numbers each of which is at most $-4$
(because each number is minus the square of a multiplicity at a singular point). Thus it is enough to show that $H_L(\calp(\call),\call)>-4$ for all $\call$.

So suppose $\calp = \calp(\call)$ is the full set of singularities of $\call$.  Say $\call$ has $d$ lines and $s$ singularities, and let $t_k$ for $k\geq 2$ denote the number of singularities of multiplicity $k$.  To apply the Hirzebruch inequality, we must first deal with the cases where either $t_d$ or $t_{d-1}$ is nonzero.  We assume $d\geq 4$ to avoid trivialities.

\textbf{Case $t_d=1$.} In this case all lines in $\call$ belong to a single pencil.  There is a single singularity of multiplicity $d$, $\calp$ is just the singular point, and $H_L(\calp,\call)=0$.

\textbf{Case $t_{d-1}=1$.} This case is called a \emph{quasi--pencil} by Hirzebruch. We have $t_2=d-1$ and all other numbers
   $t_i$ vanish. In this situation we find
   $$H_L(\calp,\call)= -2 +\frac{3}{d}.$$

   \textbf{Case $t_d=t_{d-1}=0$.}
We now use the Hirzebruch inequality and the following two obvious equalities
   \begin{equation}\label{eq:comb equalities}
   \mbox{a) } s=t_2+\ldots+t_d,\;\;\;\;\;\;\;\mbox{b) } \binom{d}{2}=\sum\limits_{k\geq 2}^d\binom{k}{2}t_k.
   \end{equation}
to estimate the $H_L$-constant
$$H_L(\calp,\call) = \frac{d^2- \sum_{k\geq 2} k^2t_k}{s}.$$ By the combinatorial equalities a), b), and the Hirzebruch inequality, we have
\begin{align*}
\sum_{k\geq 2} k^2 t_k
& = 2\sum_{k\geq 2} {k\choose 2} t_k + \sum_{k\geq 2} (k-4) t_k + 4\sum_{k\geq 2} t_k \\
&= d^2 - d - 2t_2 - t_3 + \sum_{k\geq 5} (k-4)t_k + 4s \\
&\leq d^2- 2d -t_2 - \frac{1}{4}t_3 + 4s.
\end{align*}
We conclude
$$H_L(\calp,\call) = \frac{d^2- \sum_{k\geq 2} k^2t_k}{s} \geq -4 + \frac{2d+t_2+\frac{1}{4}t_3}{s} > -4.$$
\endproof

The next corollary follows immediately from the proof of the theorem and Remark \ref{LinearCaseFullSingSet}.
We use the notation of the theorem and its proof.

\begin{corollary}\label{cor:tight bound}
If $\call$ is any configuration of lines with $H_L(\calp(\call),\call)\leq -1$ and $C$ the union of the lines, then
$$\inf_{\calp} H_L(\calp,\call)=H_L(\calp(\call),\call)=H(C).$$
Furthermore, if $t_d = t_{d-1} = 0$ then $$H_L(\calp(\call),\call) \geq -4 + \frac{2d+t_2+\frac{1}{4}t_3}{s}.$$
\end{corollary}

\begin{remark}\label{rmk:opti}
   The theorem shows that $H_L$ is a well-defined real number.  A natural question is whether there is a certain line configuration
   with ratio $H_L$ or if instead $H_L$ is only a limit of ratios from a sequence of configurations.
   Another natural question is whether $H_L=-4$.
   The least constant $H_L(\calp,\call)$
   known to us so far is $-\frac{225}{67}\approx -3.36$; see section \ref{subsection: Wiman}.
   There is also an example with $H_L(\calp,\call)=-3$ (see section \ref{subsection: Klein}).
   It is interesting to note that the infimum of values $H_L(\calp,\call)$ for arrangements $\call$ defined over the reals
   is $-3$, but $H_L(\calp,\call)>-3$ for any specific such arrangement; see \S\ref{subsec:real projective plane}.

\end{remark}
   The assumption of reducedness is essential, as we now show.
\begin{example}[The effect of fattening of the configuration]\label{ex:fat}
   Let $\call$ be a configuration of lines and $\calp$ a configuration of points, and let $k \geq 2$ be an integer.
   Let $k\call$ denote the configuration arising from $\call$ by taking
   all configuration lines with multiplicity $k$. Then it is easy to see that
   $$H_L(\calp,k\call)=k^2\cdot H_L(\calp,\call).$$
\end{example}

   We know of no examples with $H(C)\leq -4$. It is therefore reasonable to ask the following
   question.
\begin{problem}\label{pro:always -4}
   Does the lower bound $-4$ remain valid for $H(\P^2)$? Or, more directly: is
   $H(\P^2)=H_L$?
\end{problem}

   As Corollary \ref{cor:tight bound} shows, the $H_L$--constants of most interest are those
   computed using the full set of singularities of line configurations.
   In this case, if $d$ is the number of lines in a line arrangement $\call$, $s$
   the number of points of intersection of these lines (so $\calp=\calp(\call)$),
   and $m_i$ the number of lines meeting at the $i$th point,
   then (as in (b) of \eqref{eq:comb equalities}) we have
   $\binom{d}{2}=\sum_i \binom{m_i}{2}$, or $d^2-\sum_im_i^2=d-\sum_i m_i$. Thus
\begin{equation}
   H_L(\calp(\call),\call)=\frac{d-\sum_im_i}{s}=\frac{d}{s}-\overline{m},
\end{equation}
   with $\overline{m} = \frac{1}{s}\sum_{i=1}^{s} m_{i}$ being the average of the multiplicities $m_{i}$.
      If we define $m$ to be such that $d(d-1)=sm(m-1)$ (and hence $d^2/s> (m-1)^2$ or $\frac{d}{\sqrt{s}}+1>m$),
   then (using a standard fact; see Lemma \ref{standard fact}) we get an estimate
\begin{equation}
  \frac{d}{s}-m\leq H_L(\calp(\call),\call)
\end{equation}
  depending only on $d$ and $s$
   (which leads to the following less accurate but even simpler estimate
   $H_L(\calp(\call),\call)> d(\frac{1}{s}-\frac{1}{\sqrt{s}})-1>-\frac{d}{\sqrt{s}}-1$).

   Moreover, if the lines are not concurrent, then the proper transform of the lines on the blow up of $\P^2$ at the $s$ points
   are disjoint and all have negative self-intersection. The Index Theorem thus says
   (except in the trivial case of concurrent lines) that there can be at most $s$ lines; i.e., that $d\leq s$, and hence
 \begin{equation}
  -\overline{m}< H_L(\calp(\call),\call)\leq1-\overline{m}.
\end{equation}

\begin{remark}
   The above bounds hold for all ground fields in all characteristics.
   This suggests introducing notation for linear $H$-constants for configurations
   of lines defined over a specific field $\K$ other than just the complex numbers,
   namely, $H_{L,\K}$.
   For example, if one takes $\call$ to be all of the lines defined over a finite field of $q$ elements
so $\calp(\call)$ is the set of all of the points defined over that field, then
$s=d=q^2+q+1$ and $\overline{m}=m=q+1$, so $H_L(\calp(\call),\call)=-q=\frac{d}{s}-m=1-\overline{m}$.
Thus, over an algebraically closed field $\K$, $H_{L,\K}=-\infty$ if $\operatorname{char}(\K)>0$,
so the question of what is the value of $H_{L,\K}$, is of interest only if $\operatorname{char}(\K)=0$.
It is not clear what exactly the least value of $H_L(\calp(\call),\call)$ is for arrangements $\call$
defined over a fixed finite field. However, any
line arrangement $\call'$ and its singular points $\calp(\call')$
defined over a finite field of $q$ elements has $d'\leq q^2+q+1$ lines and $s'\leq q^2+q+1$ points
with multiplicity $m_i'\leq q+1$ at each point, so $s'm'(m'-1)=\sum_im_i'(m_i'-1)\leq s'q(q+1)\leq sq(q+1)$,
hence $H_L(\calp(\call'),\call')\geq \frac{d'}{s'}-m'> -m'\geq -q-1$.
\end{remark}

  \begin{lemma}\label{standard fact}
   Consider any finite set of $s$ positive integers $m_i$. Let $\overline{m}$ be the average and let
   $m$ be defined so that $\sum_im_i(m_i-1)=sm(m-1)$. Then $m\geq \overline{m}$
   with equality if and only if all $m_i$ are equal.
   \end{lemma}

   \begin{proof} If we let $c=\sum_im_i(m_i-1)/s$, then $m=(1+\sqrt{1+4c})/2$.
   As is well known and easy to prove, $\sum_im_i^2/s\geq \overline{m}^2$, with equality if and only if all $m_i$ are equal.
   Thus $c=\sum_im_i(m_i-1)/s\geq \overline{m}^2-\overline{m}$,
   so $1+4c\geq 4(\overline{m}^2-\overline{m})+1=(2\overline{m}-1)^2$,
   hence $m=(1+\sqrt{1+4c})/2\geq \overline{m}$, with equality if and only if all $m_i$ are equal.
   \end{proof}

   In those cases where the $m_i$ are all equal, we of course have $m_i=m=\overline{m}$,
   but (over the complex numbers) we know of only one such nontrivial configuration with $m>2$,
   and hence $m=3$ by the Hirzebruch inequality \eqnref{eq:Hirzebruch}.
   This configuration is the dual of the Hesse configuration,
   for which $d=9$, $s=12$ (under duality, the 9 lines are the flex points of a smooth plane cubic,
   and the 12 points are the lines through pairs of flex points).
   A generalization of this is given in the following example.

\begin{example}\label{s-elliptic}
This example is taken from \cite[p.~120]{Hir83} (see also \cite{S1868}).
Let $D$ be a smooth plane cubic curve. Pick a flex to define the group law on the points of $D$.
Let $\call$ be the lines dual to the points of a finite group $U$ of order $k$ and let $w$ be the number of
lines in $\call$ whose duals are flex points. Let $C$ be the union of the $k$ lines.
Then $t_2=k-w$ and $t_3=\frac{k(k-3)}{6}+\frac{w}{3}$, and so we get
$H(C) = -3+\frac{12k-6w}{k^2+3k-4w}$, which approaches $-3$ from above as $k\to\infty$.
\end{example}

   Since in the complex projective plane the dual Hesse configuration
   is the unique nontrivial configuration of lines with only triple intersection points that we know of,
   it is a natural problem of independent interest to wonder if this is in fact the only such configuration.
   The following problem might
   be viewed as a first step towards understanding configurations
   having only triple points. Apparently, questions revolving around the same idea,
   have been present in combinatorics and discrete geometry for some time, see \S 1.1 in \cite{SolSto13}.
\begin{problem}
   Let $\call=\left\{L_1,\ldots,L_s\right\}$ be a configuration of lines
   with only triple intersection points. Can $\call$ be equipped with
   a group structure?
\end{problem}

   Of course, the question above has an affirmative answer for the dual Hesse configuration.
   A possible way to introduce a group
   structure on a configuration $\call$ with only triple intersection points
   would be as follows. We fix one configuration line, say $L_1$ and declare it
   as the neutral element. If a line $L$ intersects $L_1$ in some point $P$,
   then the third line passing through $P$ will be $-L$. This explains the
   addition for lines $L,M$ in the same pencil as $L_1$. For the general case,
   assume that the intersection point $P=L \cap M$ does not belong to $L_1$.
   Let $N$ be the third line passing through $P$. Then we set $L+M=-N$.
   We were not able to verify if this construction does indeed lead to a group
   structure on $\call$.

\subsection{The real projective plane}\label{subsec:real projective plane}
   Families with limiting $H$-constant of $-3$, such as that in Example \ref{s-elliptic},
   can also be given over the reals, where the limiting value of $-3$ has special relevance.
   To explain this in greater detail, we now consider arrangements $\call$ of lines defined over the reals.
   The results of this section might be of independent interest in combinatorics.
   Let $C$ be the union of the lines in $\call$. If the lines are concurrent, then for any point set $\calp$
   it is easy to check that $H(C;\calp)>-1$. So suppose the lines are not concurrent.
   Then $d\geq3$ and there is Melchior's Inequality \cite{M41}
   \begin{equation}\label{eq:Melchior}
      t_2 \geq 3 + \sum_{k>3} (k-3)t_k,
   \end{equation}
   which is stronger than the Hirzebruch inequality above but valid only over $\R$ (note that for example it solves
   the Sylvester-Gallai problem: any arrangement of \emph{real} nonconcurrent lines must have
   a point where exactly two lines meet; see \cite{BorMos90} for an excellent introduction
   to this circle of ideas).
   Thus $t_2 = e + 3 + \sum_{k>3} (k-3)t_k$ for some $e \geq 0$, so
\begin{equation}
\begin{split}
H(C) &= \frac{d-\sum_{k>1} k t_k}{\sum_{k>1} t_k}> \frac{-\sum_{k>1} kt_k}{\sum_{k>1} t_k}
         = \frac{-(2e+6+2\sum_{k\geq 3}(k-3)t_k) - \sum_{k\geq 3}kt_k}{\sum_{k>1} t_k}\\
        &= \frac{-(2e+6+3\sum_{k\geq 3}(k-2)t_k)}{\sum_{k>1} t_k}= \frac{-(2e+6+3\sum_{k\geq 3}(k-2)t_k)}{e+3+\sum_{k\geq 3}(k-2)t_k}\\
        &= -3 + \frac{e+3}{e+3+\sum_{k\geq 3}(k-2)t_k} > -3.
\end{split}
\end{equation}
   Thus we have proved a real analogue of Theorem \ref{thm:main}. Examples \ref{k-gons} and \ref{Boroczky}
   thus give the following result.
\begin{theorem}[Bounded linear negativity for $\P^2(\R)$]
   We have $H_{L,\R}(\P^2(\R))=-3$.
\end{theorem}

\begin{example}\label{k-gons}
Let $\call$ be the configuration of $2k$ lines where $k$ of the lines are the sides of a regular $k$-gon
and the other $k$ lines are the lines of bilateral symmetry of the $k$-gon (i.e., angle bisectors and
perpendicular bisectors of the sides). Then it is not hard to check that
$t_k=1$ (this is the center of the $k$-gon), $t_2=k$ (these are the midpoints of the sides)
$t_3 = \binom{k}{2}$ (these are the intersections of pairs of sides with the line of symmetry between the sides of the pair)
and $t_r=0$ for $r>3$.
Thus for the union $C$ of the $2k$ lines we get $H(C)=-3+\frac{4k+6}{k^2+k+2}$. These again approach $-3$ as
$k\to\infty$.
\end{example}

\begin{example}\label{Boroczky}
Less elementary real examples which also approach $-3$ as $k\to\infty$ are given in
\cite[Property 4 of Section 6]{FP84}. The point of the examples was to show sharpness of bounds
concerning how many triple points a real line arrangement can have.
We simply note that for each $k$ divisible by 6, a line arrangement is given over the reals with
$k$ lines, $t_2= k-3$, $t_3=1+k(k-3)/6$ and $t_i=0$ for $i>3$. Thus the corresponding
$H$-constants are $-3+\frac{12k+54}{k^2+3k-12}$.
\end{example}
%

   In the next section, we discuss some other interesting
   configurations coming from unitary reflection groups and compute their linear $H$--constants.

\section{Arrangements of lines with low linear $H$--constants}\label{sec: low H-constants}
   As an alternative to asking for configurations with only triple points, one can ask for
   configurations with no double points. We know of only three kinds of line arrangements (over the complex numbers) for which
   there are no double points.

   The first one generalizes the dual of the Hesse configuration.
   Recall that the original Hesse configuration
   consists of $12$ lines passing through the flexes of a smooth plane cubic.
   The $9$ lines of its dual can be taken to be
   the linear factors of $(y^3-z^3)(x^3-z^3)(x^3-y^3)$. The generalization consists of the
   lines $\call_n$ given by the factors of $(y^n-z^n)(x^n-z^n)(x^n-y^n)$ for $n\geq3$. Urz\'ua calls the resulting
   configurations \emph{Fermat arrangements} \cite[Example II.6]{Urz08}.
   The corresponding points are the $n^2$ points
   of intersection of $x^n-z^n=0$ and $y^n-z^n=0$, together with the three coordinate vertices.
   The three coordinate vertices occur with multiplicity $n$; the other $n^2$ points
   are triple points. For these we have $H_L(\calp(\call_n),\call_n)=\frac{3n-3n-3n^2}{n^2+3}>-3$
   with $\lim_{n\to\infty} H_L(\calp(\call_n),\call_n)=-3$.

   There are only two other arrangements $\call$ with no double points that we know of,
   one due to Klein \cite{Kle79} with 21 lines (for which $H_L(\calp(\call),\call)=-3$) and another due to Wiman \cite{Wim96}
   with 45 lines (for which $H_L(\calp(\call),\call)=\frac{-225}{67}\approx -3.36$). These
   and certain subconfigurations of the Wiman configuration (see \S \ref{subsection: subconfigurations}) are  the only
   line arrangements we know of with $H_L(\calp(\call),\call)\leq -3$.

   The examples above, i.e., the Fermat, Klein and Wiman arrangements,
also are interesting for another reason.  Let $I(P) \subset \C [x,y,z]$ be the homogeneous
ideal of a point $P\in \P^2$.  Then the homogeneous ideal of a finite set of points
$\{P_1,\ldots,P_s\}\subset \P^2$ is $I = \bigcap_i I(P_i)$, and the $m$th symbolic power
of $I$ can be defined to be $I^{(m)} = \bigcap_i I(P_i)^m$.
It is typically quite rare to have a failure of containment $I^{(3)}\not\subseteq I^2$. But if $I$ is the ideal of the points of intersection
of the lines for any of these three cases, we have $I^{(3)}\not\subseteq I^2$ (see \cite{DST13, HS14}
   for the Fermat arrangements). This is also true for at least some of the configurations given in Example \ref{Boroczky} \cite{CGetal13}, so
   we suspect that the same may hold for the configurations given in Example \ref{s-elliptic},
   but we are not dwelling on this problem here.

   We hope to come back to the ideal theoretic properties related to
   configurations of lines in a separate paper in the near future.

   Now we address the two exotic configurations in more detail.

\subsection{The Klein configuration of 21 lines}\label{subsection: Klein}
The Klein configuration is a projective configuration of $21$ lines whose intersections consist of
precisely $21$ quadruple points and $28$ triple points, defined over $\mathbb{R}[\sqrt{-7}]$
(see the top entry of the table on p.~120 of \cite{Hir83}).
In this case $H_L(\calp(\call),\call)=\frac{d-\sum_im_i}{s}=\frac{21-168}{49}= -3$.
We also note that each of the 21 lines contains eight of the 49 points, four of each type.
(To see this, note that the lines are the fixed lines of the 21 involutions of the group $PSL(2,7)$ acting on $\P^2$.
There are 21 elements of order 2, which by the character table of the group
(see \url{http://brauer.maths.qmul.ac.uk/Atlas/v3/})
form a single orbit and hence the group acts transitively on the 21 lines. Thus the number of 4-points on each line is the same
as is the number of 3-points. So if $a$ is the number of 3-points per line and $b$ the number of 4-points per line,
then counting the number of pairs $(p,L)$ where $L$ is one of the 21 lines and $p$ is a 4-point on $L$, we get $21b$
since there are $b$ points on each of 21 lines, but we also get $21\cdot4$, since each point of the 21 4-points $p$ is on 4 lines. Thus $b=4$.
Similarly, $21a=28\cdot3$ so $a=4$ also. We thank the referee for the foregoing argument.)
Explicit equations for the 21 lines over the complex numbers are given in \cite{GR90}.

\subsection{The Wiman configuration of 45 lines}\label{subsection: Wiman}
The Wiman configuration \cite{Wim96} is a projective configuration of $45$ lines whose 201 intersections consist of
precisely 36 quintuple points, 45 quadruple points, and 120 triple points (see the bottom line of the table
on p.~120 of \cite{Hir83}).
In this case $H_L(\calp(\call),\call)=\frac{d-\sum_im_i}{s}=\frac{45-720}{201}=\frac{-225}{67}\approx -3.36$.
This is the most negative example we know.

   Arguing as for the Klein configuration shows that the points are equidistributed
   in this case as well. There are $4$ quintuple, $4$ quadruple and $8$ triple
   points on each of the configuration lines.

\subsection{Subconfigurations of special configurations}\label{subsection: subconfigurations}

   We close the paper by noting that the Wiman and Klein configurations also have subconfigurations $\call'\subset \call$ with highly negative constants $H_L(\calp(\call'),\call')$ by the next result.

\begin{proposition}\label{prop:subconfiguration}
Let $\call$ be a configuration of $d$ lines and let $\calp = \calp(\call)$ be the set of singularities of $\call$.  Suppose each line in $\call$ contains the same number $n$ of points of $\calp$.  Let $\call'\subset \call$ be a subconfiguration of $d'$ lines.  Then $$H_L(\calp,\call') = H_L(\calp,\call) + \frac{(d-d')(n-1)}{s}.$$
\end{proposition}
\begin{proof}
Say $\calp = \{Q_1,\ldots, Q_s\}$ and put $m_i = \mmult_{Q_i}(\call)$ and $m_i' = \mmult_{Q_i}(\call')$.  Observe that $d'(d'-1) = \sum_i m_i'(m_i'-1)$ since $\calp$ contains the singularities of $\call'$.  Also, since every line in $\call$ contains $n$ points of $\calp$, we have $(d-d')n+ \sum_i m_i' = \sum_i m_i$.  We conclude $$H_L(\calp,\call') = \frac{d' - \sum_i m_i'}{s} = \frac{d' + (d-d')n- \sum_i m_i}{s} =H_L(\calp,\call) + \frac{(d-d')(n-1)}{s},$$ as claimed.
\end{proof}

\begin{remark}
With the hypotheses of the proposition, suppose $H_L(\calp(\call),\call')\leq -1$.  Then we have an inequality $$H_L(\calp(\call'),\call') \leq H_L(\calp(\call),\call') = H_L(\calp(\call),\call)+\frac{(d-d')(n-1)}{s}$$ as in Remark \ref{LinearCaseFullSingSet}.  However, for ``large'' subcollections $\call'\subset \call$ we often have an equality $\calp(\call') = \calp(\call)$.
\end{remark}

\begin{example}
Here are some explicit applications of Proposition \ref{prop:subconfiguration} to the Wiman configuration $\call$.  Let $\call'\subset \call$ be a subcollection and let $\call'' = \call - \call'$ be its complement.

If $\call''$ is a single line, then $H_L(\calp(\call'),\call') = -\frac{220}{67} \approx -3.28.$

Next suppose $\call''$ is a pair of lines.  Then $H_L(\calp(\call),\call') = -\frac{215}{67} \approx -3.21$ regardless of what pair of lines $\call''$ is.  If the lines in $\call''$ meet at a point of multiplicity at least $4$ in $\call$, then $\calp(\call') = \calp(\call)$ and $H_L(\calp(\call'),\call') = -\frac{215}{67}$.  On the other hand, if they meet at a point $Q$ of multiplicity $3$ in $\call$ then $\call'$ has only $200$ singularities, $\calp(\call) = \calp(\call') \cup \{Q\}$, and $\mmult_Q(\call') = 1$.
Starting from the fact that there are $4$ quintuple, $4$ quadruple and $8$ triple points on each of the lines of the Wiman configuration, we find that
$\call'$ has 14 double points, 113 triple points, 45 quadruple points and 28 quintuple points, so $H_L(\calp(\call'),\call') = -\frac{161}{50} = -3.22$.

Similar computations of the constants $H_L(\calp(\call'),\call')$ can be performed as the size of $\call''$ grows, but the combinatorics of $\call$ obviously will play an important role in determining all the constants obtainable in this way.
\end{example}

\paragraph{Acknowledgments}
Th.~Bauer was partially supported by DFG grant BA~1559/6--1,
B. Harbourne was partially supported by NSA grant H98230-13-1-0213,
J.~Huizenga was partially supported by a National Science Foundation Mathematical Sciences Postdoctoral Research Fellowship,
A.~Lundman was supported by the V.R. grant NT:2010-5563,
T.~Szemberg was partially supported by NCN grant UMO-2011/01/B\-/ST1/04875.
The authors thank the referee for his very helpful comments.
We thank also the Mathematisches Forschungsinstitut Oberwolfach for hosting a workshop in February 2014
where the work presented in this paper was started.



\bigskip
   Tho\-mas Bau\-er,
   Fach\-be\-reich Ma\-the\-ma\-tik und In\-for\-ma\-tik,
   Philipps-Uni\-ver\-si\-t\"at Mar\-burg,
   Hans-Meer\-wein-Stra{\ss}e,
   D-35032~Mar\-burg, Germany.

\nopagebreak
   \textit{E-mail address:} \texttt{tbauer@mathematik.uni-marburg.de}

\bigskip
   Sandra Di Rocco,
   Department of Mathematics, KTH, 100 44 Stockholm, Sweden.

\nopagebreak
   \textit{E-mail address:} \texttt{dirocco@math.kth.se}

\bigskip
   Brian Harbourne,
   Department of Mathematics, University of Nebraska-Lincoln, Lincoln, NE, 68588, USA

\nopagebreak
   \textit{E-mail address:} \texttt{bharbourne1@unl.edu}

\bigskip
   Jack Huizenga,
   Department of Mathematics, Statistics, and Computer Science, University of Illinois at Chicago, Chicago, IL 60607, USA

\nopagebreak
   \textit{E-mail address:} \texttt{huizenga@uic.edu}

\bigskip
   Anders Lundman,
   Department of Mathematics, KTH, 100 44 Stockholm, Sweden.

\nopagebreak
   \textit{E-mail address:} \texttt{alundman@kth.se}

\bigskip
   Piotr Pokora,
   Instytut Matematyki UP,
   Podchor\c a\.zych 2,
   PL-30-084 Krak\'ow, Poland.

\nopagebreak
   \textit{E-mail address:} \texttt{piotrpkr@gmail.com} \\ \nopagebreak
Current address:
   Albert-Ludwigs-Universit\"at Freiburg,
   Mathematisches Institut, D-79104 Freiburg, Germany.

\bigskip
   Tomasz Szemberg,
   Instytut Matematyki UP,
   Podchor\c a\.zych 2,
   PL-30-084 Krak\'ow, Poland.

\nopagebreak
   \textit{E-mail address:} \texttt{tomasz.szemberg@gmail.com} \\ \nopagebreak
Current address:
   Albert-Ludwigs-Universit\"at Freiburg,
   Mathematisches Institut, D-79104 Freiburg, Germany.


\end{document}